\documentclass[english,10pt]{article}
\usepackage{amsmath}
\usepackage{amsthm}
\usepackage{amsfonts}
\usepackage[named]{algo}
\usepackage[latin9]{inputenc}
\usepackage{amssymb}
\usepackage{babel}
\usepackage{tkz-graph}
\usepackage[T1]{fontenc}
\usepackage[latin9]{inputenc}
\usepackage{units}
\usepackage{amssymb}
\usepackage{esint}
\usepackage{babel}
\usepackage{color}

\DeclareFontEncoding{LGR}{}{}
\DeclareTextSymbol{\~}{LGR}{126}

\newtheorem{theorem}{Theorem}[section]

\begin{document}

\title{Combinatorial Constructions for Sifting Primes and Enumerating the Rationals}

\author{Edinah K. Gnang 
\thanks{Department of Computer Science, Rutgers University, Piscataway, NJ
08854-8019 USA%
} 
\and 
Chetan Tonde 
\thanks{Department of Computer Science, Rutgers University, Piscataway, NJ
08854-8019 USA%
}
\and
\\
\it{In memory of Herbert Wilf}
}
\maketitle
\begin{abstract}
{
We describe a combinatorial approach for investigating  
properties of rational numbers. The overall approach rests on 
structural bijections between rational numbers and familiar 
combinatorial objects, namely rooted trees. We emphasize that
such mappings achieve much more than enumeration 
of rooted trees.\\
We discuss two related structural bijections. 
The first corresponds to a bijective map between 
integers and rooted trees. The first bijection also suggests a new algorithm 
for sifting primes. The second bijection extends the first one in order to map
rational numbers to a family of rooted trees. The second bijection suggests a new 
combinatorial construction for generating reduced rational numbers, 
thereby producing refinements of the output of the Wilf-Calkin[1] Algorithm.}
\end{abstract}

\begin{figure} 
\center
\scalebox{0.5}{
\begin{tikzpicture}[level/.style={sibling distance=30mm/#1}]
\node [circle,draw] (z){$r$}
  child { node [circle,draw] (a) {$v_1$}}
  child {node [circle,draw] (b) {$v_2$}
        child {node [circle,draw] (c) {$v_3$}}
        child {node [circle,draw] (d) {$v_4$}   
                child [grow = down] { node [circle,draw] (e) {$v_5$}}
                }
};
\end{tikzpicture}
}
\caption{Typical example of a rooted tree. Root node: $r$, leaves: $\{v_1,v_3,v_5\}$,\newline 
(parent,child) pair: $(v_2, v_3)$, connected path: $r \rightarrow v_2\rightarrow v_4 \rightarrow v_5$.}
\end{figure}
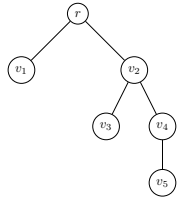

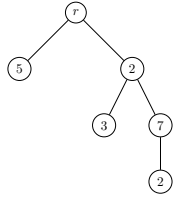
\begin{figure} 
\center
\scalebox{0.5}{
\begin{tikzpicture}[level/.style={sibling distance=30mm/#1}]
\node [circle,draw] (z){$r$}
  child { node [circle,draw] (a) {$5$}}
  child {node [circle,draw] (b) {$2$}
        child {node [circle,draw] (c) {$3$}}
        child {node [circle,draw] (d) {$7$}     
                child [grow = down] { node [circle,draw] (e) {$2$}}
                }
};
\end{tikzpicture}
}
\caption{A valid label assignment with labels = 2,3,5 and 7.}
\end{figure}

\begin{figure} 
\center
\scalebox{0.5}{
\begin{tikzpicture}[level/.style={sibling distance=30mm/#1}]
\node [circle,draw] (z){$r$}
  child { node [circle,draw] (a) {$2$}}
  child {node [circle,draw] (b) {$3$}
        child {node [circle,draw] (c) {$5$}}
        child {node [circle,draw] (d) {$5$}     
                child [grow = down] { node [circle,draw] (e) {$7$}}
                }
};
\end{tikzpicture}
}
\caption{An invalid label assignement with labels = 2,3,5 and 7.}
\end{figure}

\section{The Combinatorics}
The word "\emph{Combinatorics}" here mostly refers to the combinatorics of \emph{trees}, more specifically labeled \emph{rooted trees} see Fig[1]. The defining property of a tree is that precisely one path connects any two vertices of the tree. 
At various parts of the discussion we will assign labels to the vertices of the trees. \emph{Rooted} trees have each a special vertex (typically labeled $r$)
referred to as the \emph{root} of the tree. Moreover a collection of rooted trees is called a
\emph{planted forest}. For convenience in the subsequent discussion the word "\emph{tree}" will 
refer to a rooted tree while the word "\emph{forest}" will refer to a planted forest.\\
Each vertex $v \ne \, root \,$, is adjacent to a single vertex whose distance to the root is smaller than 
the distance between $v$ and the root. (The distance between two vertices here refers to the 
number of edges on the unique path connecting the two vertices.) Such a vertex is called the 
\emph{parent} of $v$ and all other vertices adjacent to $v$ are called \emph{children} vertices of 
$v$.\\
The analogy to human family relations is extended to include notions such as \emph{siblings} and 
\emph{grandparent} relationships between vertices. We assume that such relationships between 
vertices are unambiguous within the current context of trees. 
Vertices with no \emph{children} are called \emph{leaves} of the tree.\\
Finally, a label assignment to the vertices of a given tree is considered valid, 
if the root is labeled $r$ and no two sibling vertices are assigned the same label. See Fig[2,3].

\subsection{Operations on trees and forests}

\subsubsection{Grafting trees and forests}

The first operation on trees that we introduce here is the \emph{grafting} operation.\\
Let $T_{\alpha}$ and $T_{\beta}$ denote two rooted trees. We say that the 
tree $T_{\gamma}$ results from the grafting of $T_{\alpha}$ onto $T_{\beta}$ denoted 
\begin{equation}
T_{\gamma}=T_{\alpha}\:\wedge\: T_{\beta}
\end{equation}
iff $T_{\alpha}$ and $T_{\beta}$ are disjoint rooted subtrees of $T_{\gamma}$
and most importantly $T_{\gamma}$ has no other rooted subtree disjoint
from the subtrees $T_{\alpha}$ and $T_{\beta}$ as illustrated in Fig[4].\\
Note that, it follows from the definition that the grafting operation is commutative 
\begin{equation}
T_{\alpha}\:\wedge\: T_{\beta}=T_{\beta}\:\wedge\: T_{\alpha}.
\end{equation}
The single--vertex rooted tree (having only the root vertex), can be thought to 
represent the neutral (or identity) element for the grafting operation. We write 
\begin{equation}
root\:\wedge\: root=root
\end{equation}
and 
\begin{equation}
T_{\alpha}\:\wedge\: root=root \:\wedge\: T_{\alpha}=T_{\alpha}
\end{equation}
The tree grafting operation induces a natural algebra on forests described below. \\
Let $\mathcal{F}=\left\{T_{k}\right\}_{1\le k \le n}$ and $\mathcal{G}=\left\{ T_{l}^{\prime}\right\} _{1\le l\le m}$
denote two forests. The grafting of the forest $\mathcal{F}$
onto the forest $\mathcal{G}$ amounts to creating a new forest made of trees resulting from the 
grafting of all possible pairs of trees $\left(T_{u},T_{v}^{\prime}\right)$ such that  
\begin{equation}
\left(T_{u},T_{v}^{\prime}\right) \in \mathcal{F}\times \mathcal{G}.
\end{equation}
Hence,
\begin{equation}
\mathcal{F}\wedge\mathcal{G}=\left\{ T_{k}\wedge T_{l}^{\prime}\right\}_{
\begin{array}{c}
1\le k\le n, 1\le l\le m.
\end{array}}
\end{equation}

\begin{figure} 
\center
\scalebox{0.6}{
\begin{tikzpicture}[level/.style={sibling distance=30mm/#1}]
\node [circle,draw] (z){$r$}
  child {node [circle,draw] (a) {${\color{blue}\boldsymbol{2}}$}}
  child {node [circle,draw] (b) {${\color{blue}\boldsymbol{3}}$}};
\end{tikzpicture} 
{\Huge $\bigwedge$}
\begin{tikzpicture}[level/.style={sibling distance=30mm/#1}]
\node [circle,draw] (z){$r$}
  child [grow = down] { node [circle,draw] (b) {${\color{red}\boldsymbol{5}}$}
        child {node [circle,draw] (c) {${\color{red}\boldsymbol{7}}$}}
        child {node [circle,draw] (d) {${\color{red}\boldsymbol{11}}$}}
        };      
\end{tikzpicture}
{\Huge =}
\begin{tikzpicture}[level/.style={sibling distance=30mm/#1}]
\node [circle,draw] (z){$r$}
  child { node [circle,draw] (a) {${\color{blue}\boldsymbol{2}}$}}
  child { node [circle,draw] (a) {${\color{blue}\boldsymbol{3}}$}}
  child {node [circle,draw] (b) {${\color{red}\boldsymbol{5}}$}
        child {node [circle,draw] (c) {${\color{red}\boldsymbol{7}}$}}
        child {node [circle,draw] (d) {${\color{red}\boldsymbol{11}}$}}
};
\end{tikzpicture}
}
\caption{Grafting operation on trees.}
\end{figure}
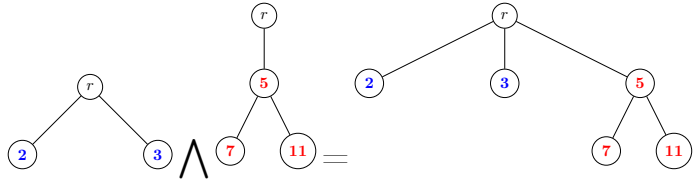

\subsubsection{Raising forests}

Let $\mathcal{F}=\left\{ T_{k}\right\} _{1\le k\le n}$ denote a forest
and $T_{\alpha}$ denote a tree. The operation of raising the forest 
$\mathcal{F}$ by the tree  $T_{\alpha}$ noted 
\begin{equation}
\mathrm{\mathcal{R}}\left(T_{\alpha},\mathcal{F}\right)
\end{equation}
consists in substituting every tree $T_{k}$ in the forest $\mathcal{F}$
with an extended tree $T_{k}^{\prime}$ constructed by rooting
the tree $T_{k}$ at every leaf of $T_{\alpha}$ as illustrated in Fig[5].\\
\begin{figure} 
\center
\scalebox{0.5}{
{\Huge $\mathcal{R}($}
\begin{tikzpicture}[level/.style={sibling distance=30mm/#1}]
\node [circle,draw] (z){$r$}
  child {node [circle,draw] (a) {${\color{blue}\boldsymbol{2}}$}};
\end{tikzpicture} 
{\Huge ,}
{\Huge \{}
\begin{tikzpicture}[level/.style={sibling distance=30mm/#1}]
\node [circle,draw] (z){$r$}
  child [grow = down] { node [circle,draw] (b) {${\color{red}\boldsymbol{3}}$}
        child {node [circle,draw] (c) {${\color{red}\boldsymbol{5}}$}}
        child {node [circle,draw] (d) {${\color{red}\boldsymbol{7}}$}}
        };      
\end{tikzpicture}
{\Huge ,}
\begin{tikzpicture}[level/.style={sibling distance=30mm/#1}]
\node [circle,draw] (z){$r$}
  child { node [circle,draw] (a) {${\color{green}\boldsymbol{5}}$}}
  child { node [circle,draw] (b) {${\color{green}\boldsymbol{7}}$}};
\end{tikzpicture}
{\Huge \} )} 
\newline {\Huge = } 
{\Huge \{ } 
\begin{tikzpicture}[level/.style={sibling distance=30mm/#1}]
\node [circle,draw] (z){$r$}
  child [grow = down] { node [circle,draw] (a) {${\color{blue}\boldsymbol{2}}$}
        child [grow = down] { node [circle,draw] (b) {${\color{red}\boldsymbol{3}}$}
                child {node [circle,draw] (c) {${\color{red}\boldsymbol{5}}$}}
                child {node [circle,draw] (d) {${\color{red}\boldsymbol{7}}$}}
                }
};
\end{tikzpicture}
{\Huge ,}
\begin{tikzpicture}[level/.style={sibling distance=30mm/#1}]
\node [circle,draw] (z){$r$}
  child [grow = down] {node [circle,draw] (b) {${\color{blue}\boldsymbol{2}}$}
        child {node [circle,draw] (c) {${\color{green}\boldsymbol{5}}$}}
        child {node [circle,draw] (d) {${\color{green}\boldsymbol{7}}$}}
};
\end{tikzpicture}
{\Huge \} } 
        }
\caption{Raising operation on forests.}
\end{figure}
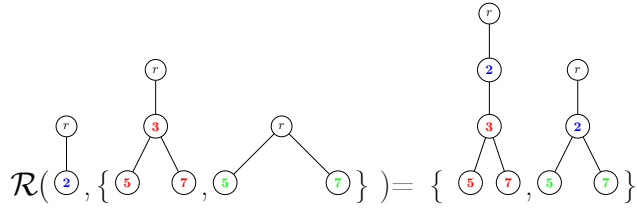
We note that the raising operator is not commutative. 
It also follows from the definition that for any tree $T_{\alpha}$ we have
\begin{equation}
\mathrm{\mathcal{R}}\left(root,T_{\alpha}\right) = root.
\end{equation}
Finally, for convenience we adopt the convention that  
\begin{equation}
\mathrm{\mathcal{R}}\left(T_{\alpha},root \right) =  T_{\alpha}.
\end{equation}

\subsection{Generating the forest of all labeled trees of height bounded by $h$ using $n$ labels.}

We recall that a label assignment to vertices of a tree is considered valid if the root is unlabeled 
and no two sibling vertices of the tree are assigned the same label. 
We assume here that the label set is of size $n$, and furthermore, if we restrict 
our attention to the trees of height bounded by a positive integer $h$,
it is clear that the set of such trees is finite.\\ 
We now describe a construction which produces the set of trees described above as a forest. 
For this purpose we consider the following recursion.
\begin{equation}
G_{0}=\left\{ root \right\}
\end{equation}
and
\begin{equation}
G_{i}=\bigwedge_{0\le k\le n-1}\left\{ root,\:\mathrm{\mathcal{R}}\left(T_{k},G_{i-1}\right)\right\},
\end{equation}
that is,
\begin{equation}
G_{i}=\left\{ root,\:\mathrm{\mathcal{R}}\left(T_{0},G_{i-1}\right)\right\}\wedge \cdots \wedge \left\{ root,\:\mathrm{\mathcal{R}}\left(T_{n-1},G_{i-1}\right)\right\}
\end{equation}
where $T_{k}$ denotes the tree made of the root and an additional 
vertex who is assigned the $k^{th}$ label from our labeling set
$L=\left\{0, 1,\cdots,(n-1)\right\} $.\\
For illustration purposes we express $G_1$ 
\begin{equation}
G_{1}=\bigwedge_{0\le k\le n-1}\left\{ root,\: T_{k}\right\}. 
\end{equation}
\begin{equation}
=\left\{ root,\: T_{0}\right\}\wedge \cdots \wedge \left\{ root,\: T_{n-1}\right\}. 
\end{equation}
By construction we have that for $1 \le i \le h$ the trees in the forest $G_{i}$
are distinct, validly labeled and of depth bounded above by $i$. Furthermore 
\begin{equation}
G_{s} \varsubsetneq G_{s+1}
\end{equation}
It also follows from the definition of the recurrence that $G_{1}$ has $2^{n}$
elements and that the number of trees in the forest $G_{i}$ is prescribed by the recurrence relation 
\begin{equation}
\left|G_{i}\right|=\left(1 + \left |G_{i-1}\right|\right)^{n}
\end{equation}

\noindent \textbf{Theorem 1:} The forest $G_{h}$ contains all trees
with valid vertex label assignment and of height bounded by $h$.\\

\textbf{Proof :} The theorem is proved by first observing that
the validly labeled trees of height bounded by $h$ correspond 
to rooted subtrees of the complete $n$-ary tree of height $h$. 
A simple counting argument reveals that the number of such trees 
is determined by the recurrence relation 
\begin{equation}
S_0 = 1
\end{equation}
and
\begin{equation}
\forall \, 1 \le i \le h, \; S_{i} = \left(1 + S_{i-1} \right)^n 
\end{equation}
where $S_i$ denote the number of distinct rooted subtrees of the complete $n$-ary
trees with height bounded by $i$. We therefore conclude the proof by observing that 
\begin{equation}
\forall \, 0 \le i \le h, \; \left|G_{i}\right| = S_i . \, \square 
\end{equation}

\section{Applications to Number Theory}
We discuss here the structural bijection between the set of integers 
and the set of valid prime--labeled trees. 
The mapping of an arbitrary valid tree (labeled with primes) to an 
integer, is established by equating the sibling relationship between 
vertices to integer multiplication and also equating for non-root 
vertices the parenthood relation to integer exponentiation. 
The process is therefore algorithmic and will be referred to 
as evaluation of a tree to an integer. We may point out that the 
evaluation process is recursive and insist that the evaluation be 
always initiated at the leaves because of the non-associativity 
of integer exponentiation.\\
On the other hand the mapping of an arbitrary integer to a valid  
prime--labeled tree immediately follows from recursively applying the
fundamental theorem of arithmetics to the powers of the prime factors.
The single vertex tree (only having the root as a vertex) is associated 
to the integer $1$. It therefore follows that the mapping between trees 
and integers is bijective.

\subsection{Combinatorial Prime Sieve Algorithm}

Modern sieve theory focuses on providing accurate estimates for the number of primes 
in some integer interval. In contrast, as an application of the combinatorial 
framework described above we discuss a variation of the classical sieve 
of Eratosthenes.
We recall some notation convention.\\
$\left[ n \right]$ denote the set of consecutive integers from $1$ to $n$,\\
$\mathbb{P}$ denotes the set of prime numbers \\
$\mathbb{N}$ denotes the set of strictly positive integers.\\
\noindent 
\newpage
\begin{algorithm}{Combinatorial Prime Sieve}{
\label{algo:combinatorial_prime_seive}
\qinput{For $q\in \mathbb{P}$ the set $\mathbb{P} \cap \left[ q \right]$}
\qoutput{The set $\mathbb{P} \cap \left([2q]\backslash [q]\right)$}}
$G_{out} \leftarrow \left\{ \emptyset \right\}$(*Output list*)\\
$G_{rest} \leftarrow \left\{ \emptyset \right\}$(*restricted list of composites*) \\
$G_{old} \leftarrow \left\{ \emptyset \right\} $(*old extended list of integers*)\\
$G_{new}  \leftarrow \bigwedge_{0\le k\le n-1}\left\{ root, T_k \right\}$ (*new extended list of integers*)\\
\qwhile ($\exists T \in \left(G_{new} \backslash G_{old}\right) \, \,s.t.\,\, evaluation(T) \leq 2q$ ) \\
        \qdo    $G_{old} \leftarrow G_{new}$ \\
        $G_{new} \leftarrow \bigwedge_{0\le k\le n-1}\left(root,\mathcal{R}(T_k,G_{old})\right)$ \\\\
\qfor $i \in [1, \ldots ,\left|\left(G_{new} \backslash G_{old}\right)\right| ] $ \\
        \qdo \\
        \qif $\left(q< evaluation(\left(G_{new} \backslash G_{old}\right)[i]) < 2q \right)\,\, and \,\, \left(\left(G_{new} \backslash G_{old}\right)[i]\notin G_{rest}\right)$\\(*found a new composite*)\\
        \qthen $G_{rest}\centerdot orderedInsert\{\left(G_{new} \backslash G_{old}\right)[i]\}$ 
( *added the new composite to restricted the list*) \qfi \qrof \\
\qfor $i \in [1, \ldots ,|G_{rest}|-1 ]$ \\
        \qdo \\
        \qif ($evaluation(G_{rest}[i+1]) == evaluation(G_{rest}[i])+2$) (*found a new prime*)  \\
        \qthen $G_{out} \leftarrow  G_{out}\centerdot Append\{evaluation(G_{rest}[i]+1)\} $ (*add it to the list*) \qfi \qrof \\\\
\qreturn $G_{out}$ 
\end{algorithm}

\begin{theorem}
$\forall q\in \mathbb{P} $ the algorithm above will produce all the primes in the range $(q,2q)$ 
\end{theorem}

\begin{proof}
The correctness of the algorithm follows from the well established fact that for all $n \ge 1$,
there exist some prime number $p$ with $n<p \le2n$ [2], in addition to two other facts. The first 
of which is that in the range $(q,2q)$ there is no composite whose corresponding tree 
has a vertex labeled with a prime greater than $q$. The reason being that $2q$ is the
largest of the composite less than or equal to $2q$, in other words the smallest composite 
admiting a prime $p$ greater than $q$  as a vertex label in the associated tree would be of 
the form $2p$ and consequently necessarily greater than $2q$. 
The second fact is that all the composites whose corresponding trees are labeled with primes 
less than or equal to $q$ are determined by the combinatorial algorithm described in section 1.2.
Furthermore, the sought--after primes in the range $(q,2q)$ will be uncovered by identifying 
trees in the resulting forest which evaluate to nearest composites whose difference equals $2$.
\end{proof}

\subsection{Re-enumerating the rationals}
Georg Cantor was the first to establish the suprising fact that the set $\mathbb{Q}^+$ is 
countably infinite. 
In \cite{N_CalK_H_Wilf} Neil Calkin and Herbert Wilf introduced a sequence which listed 
the elements of $\mathbb{Q}^+$ so as not to include duplicates in the sequence. 
We discuss here another sequence for listing the rationals so as not to include 
duplicates. Our sequence follows from the integer combinatorial encoding described above. 

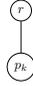
\begin{figure} 
\center
\scalebox{0.5}
{
\begin{tikzpicture}
\node [circle,draw] (z){$r$}
                child {node [circle,draw] (a) {$p_k$}};
\end{tikzpicture}
}
\caption{tree $T_k$.}
\end{figure}

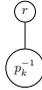
\begin{figure} 
\center
\scalebox{0.5}
{
\begin{tikzpicture}
\node [circle,draw] (z){$r$}
                child {node [circle,draw] (a) {$p^{-1}_k$}};
\end{tikzpicture}
}
\caption{tree $T_{k^{-1}}$.}
\end{figure}

The proposed construction for listing elements of $\mathbb{Q}^+$ has the benefit of providing 
explicit control over the subset of the prime numbers used to express the rationals in the 
sequence.\\  
Let $T_k$ denote the tree made of a root vertex and an additional vertex labeled with the $k^{th}$ 
prime (see Fig[6]), and let $T_{k^{-1}}$ denote the tree made of a root vertex and an additional vertex labeled with the inverse of the $k^{th}$ 
prime (see Fig[7]). Building on the recurrence for $G_i$ discussed in 1.2 we write the recurrence 
\begin{equation}
G_{0}=\left\{ root \right\}
\end{equation}
and
\begin{equation}
G_{i}=\bigwedge_{ k \in \mathbb{N}}\left\{ root,\:\mathrm{\mathcal{R}}\left(T_{k},G_{i-1}\right)\right\}. 
\end{equation}
We  use the recurrence for $G_i$  to create a new $H_{i}$ recurrence defined by
\begin{equation}
H_{i}=\bigwedge_{k\in \mathbb{N}} \left\{ \mathbf{\mathcal{R}}\left(T_{k^{-1}},\: G_{i}\right),root ,\mathbf{\mathcal{R}}\left(T_{k},G_{i} \right)\right\}.
\end{equation}
We also have 
\begin{equation}
H_{i}\varsubsetneq H_{i+1}.
\end{equation}
The trees in the forest $H_{i}$ evaluate to distinct rational numbers in reduced form and the rooted trees in the 
forest $H_{i}$ all have depth less than or equal to $i$. The trees which evaluate to non-integers 
have special vertices that are prime--inverse labeled and attached to the root. Furthermore, it
follows from the fact that siblings have different labels that corresponding rational numbers are
in their reduced form.\\
Hence as a corollary of \emph{Theorem}1 the trees in forest $H_{\infty}$ bijectively maps to $\mathbb{Q}^+$.

\section{Acknowledgments}
The authors are grateful to Professor Doron Zeilberger, Professor
Henry Cohn, Professor Mario Szegedy, Professor Vladimir Retakh
and Jules Lambert for insightful discussions and suggestions.
The authors are greatful to Jeanine Sedjro for proof reading 
the manuscript. The first author was partially supported by 
the National Science Foundation grant NSF-DGE-0549115
and Microsoft Research New England.


\begin{thebibliography}{References}
\bibitem{N_CalK_H_Wilf} N. Calkin \& H. Wilf, \newblock Recounting the rationals. \newblock {\em Amer. Math. Monthly}, 107:pp.360--363,
2000.
\bibitem{Aigner-Ziegler09} M. Aigner;G. Ziegler (2009). Proofs from THE BOOK (4th ed.). Berlin, New York: Springer-Verlag. ISBN 978-3-642-00855-9.
\bibitem{Edinah K. Gnang} Edinah K. Gnang, Experimental Number Theory, Part I : Tower Arithmetic (Preprint) arXiv:1101.3026v1 [math.NT]
\end{thebibliography}
\end{document}